\documentclass{article}

\usepackage[utf8]{inputenc}
\usepackage{amsfonts}
\usepackage{amsmath}
\usepackage{amsthm}
\usepackage{blindtext}
\usepackage{graphicx}
\usepackage[numbers]{natbib}
\usepackage{amssymb}
\usepackage{mathtools}
\usepackage{stmaryrd}
\usepackage{tikz-cd}

\newtheorem{theorem}{Theorem}[section]
\newtheorem{lemma}{Lemma}[section]
\newtheorem{corollary}{Corollary}[section]

\newtheorem{proposition}{Proposition}[section]
\theoremstyle{definition}

\title{Pontryagin Duality for Modules over Compact Discrete Valuation Rings}
\author{by Milo Moses}
\begin{document}

\maketitle

\begin{abstract}
We establish an analogue of Pontryagin duality for modules over compact discrete valuation rings $R$. Namely, we define the dual of a topological $R$ module to be its continuous $R$-module homomorphisms into $K/R$, the quotient module of the fraction field by its ring of integers. It is established that for locally compact $R$-modules the double dual map is an isomorphism and homeomorphism. Additionally, given a non-topological $R$-module a canonical topology is constructed, uniquely defined so that the double dual map will be injective and continuous. Finally, the functor assigning a module to itself equipped with canonical topology is shown to be fully faithful, allowing one to recontextualize the topological statements in purely algebraic forms.
\end{abstract}
\tableofcontents

\newcommand{\RR}{\mathbb{R}}
\newcommand{\HH}{\mathbb{H}}
\newcommand{\NN}{\mathbb{N}}
\newcommand{\QQ}{\mathbb{Q}}
\newcommand{\CC}{\mathbb{C}}
\newcommand{\FF}{\mathbb{F}}
\newcommand{\ZZ}{\mathbb{Z}}
\newcommand{\TT}{\mathcal{T}}
\newcommand{\mm}{\mathfrak{m}}
\newcommand{\pp}{\mathfrak{p}}
\newcommand{\Hom}{\mathrm{Hom}}
\newcommand{\Frac}{\mathrm{Frac}}
\newcommand{\res}{\mathrm{res}}
\newcommand{\HM}{\mathrm{HM}}
\newcommand{\Char}{\mathrm{char}}
\newcommand{\LCM}{\mathrm{LCM}}
\newcommand{\Gal}{\mathrm{Gal}}
\newcommand{\Kbar}{\overline{K}}
\newcommand{\st}{\,\,\mathrm{s.t}\,\,}
\newcommand{\FFqx}{\FF_q\llbracket x \rrbracket}

\pagebreak

\section{Introduction}

The modern field of Harmonic Analysis has its roots in the 1930s, with Lev Pontryagin's generalization of the Fourier transform to the more general setting of locally compact abelian groups. Namely, he showed certain Haar integrals relate the space of functions on a locally compact abelian group $M$ to the space of functions on a dual group $M^{\wedge}$, defined to be the space of continuous group morphisms into the circle group $\RR/\ZZ$.

In the 1950s and 60s, the burgeoning field of $p$-adic analysis was gaining steam with the proof of Mahler's theorem on Newton series \cite{mahler1958interpolation}, and Kubota-Leopoldt's introduction of $p$-adic L-functions \cite{kubota1964p}. This led to the very natural question of introducing a notion of Fourier transformation in the $p$-adic setting. The main problem was that $\ZZ_p$ does not admit a $\QQ_p$-valued Haar measure. This was fixed in the 70s, when C. F. Woodcock introduced a best-possible $p$-adic valued invariant integral on $\ZZ_p$ \cite{woodcock1974invariant}.

Later that year Woodcock introduced a $p$-adic Fourier theory on $\ZZ_p$ \cite{woodcock1974fourier}. This was extended to a $p$-adic Fourier theory on $\QQ_p$ in G Borm's thesis \cite{borm1988p}. Early works of Amice \cite{amice1964interpolation} and modern generalizations such as \cite{schneider2001p} also treat the case of $\QQ_p$ in much detail. Nobody has achieved Woodcock's original dream, however, that these results ``might be a special case of a `Harmonic Analysis' for the category of locally compact $\ZZ_p$-modules".

In this paper we reinterpret classical Pontryagin duality in a strongly $p$-adic fashion, which allows us to prove strong results about the category of locally compact $\ZZ_p$-modules. These are potentially useful not only for the development of a $p$-adic Harmonic Analysis but also in other fields where $\ZZ_p$-modules are present. For example, the Iwasawa theory of Elliptic curves is an extremely important field in modern mathematics. The $p$-adic L functions of Elliptic curves are expressed as characteristic ideals of the duals of Selmer groups. This paper allows one to understand duals of Selmer groups more properly. In particular, one will know that the double dual will be an isomorphism even when the Selmer group is not cofinitely generated.

This paper performs all of these computations in the more general setting of modules over a compact Discrete Valuation Ring (DVR). Compact DVRs include not only rings of integers of finite extensions of $\QQ_p$, but also completed polynomial rings over finite fields $\FFqx$. It has long been understood that modules over these rings had the potential for better duality theories than most.

The first results in this direction came with Kaplansky's 1952 treatment of discrete modules over compact DVRs \cite{kaplansky1953dual}. He proved that discrete modules are dual to compact modules, and that applying duality twice gives the identity. More work was done in this area in \cite{fuchs1978duality}, and the paper \cite{brumer1966pseudocompact} gave related results about discrete modules over some pseudocompact algebras. No paper, however, is able to deal with general modules over compact DVRs.

It is important to note that these is a divide in the literature between papers that define the dual of an $R$-module to be its continuous $\ZZ$-module homomorphisms into $\RR/\ZZ$, and those which define the dual of a module to be its continuous $R$-module homomorphisms into $K/R$, where $K$ is the field of fractions of $R$. This paper is the first to remark that these two spaces are functorially isomorphic when $R$ is compact DVR.

\section{Statement of Results}

The primary objects of interest in this paper are Hausdorff topological modules over compact Discrete Valuation Rings (DVRs). A DVR is a local ring whose maximal ideal is principal. Such rings have a natural metric topology induced by their valuation, with a basis of neighborhoods of $0$ given by powers of the maximal ideal. By compact DVR we refer to those rings which are compact spaces when considered with this topology. It is a standard result that a DVR is compact if and only if it is complete and has finite residue field: see e.g., \cite{serre2013local} Chapter 2 § 1 Proposition 1 for a reference.

A topological module over a topological ring $A$ refers to an $A$-module $M$ endowed with a Hausdorff topology such that the addition map $M\times M\xrightarrow{} M$, the inversion map $M\xrightarrow{}M$, and the scalar multiplication map $A\times M\xrightarrow{} M$ are all continuous. We note that in general giving $M$ the discrete topology will not necessarily produce a topological module, as the map $A\times M\xrightarrow{} M$ need not be continuous.

Locally compact topological modules are of particular interest due to the storied history of locally compact topological abelian groups, i.e., $\ZZ$-modules. They were first studied in \cite{pontrjagin1934theory}, where the following famous duality theorem was proved:

\begin{theorem}\label{Pontryagin Duality Theorem} Let $M$ be a locally compact abelian group. Define $\widehat{M}$ to be the space of continuous group homomorphisms $M\xrightarrow{} \RR/\ZZ$ endowed with the compact-open topology. Then, $\widehat{M}$ is locally compact and the natural map

\begin{align*}
M&\xrightarrow{} \widehat{\widehat{M}}\\
m&\mapsto \left(\varphi \mapsto \varphi(m)\right)
\end{align*}

is an isomorphism of groups and a homeomorphism of topological spaces.
\end{theorem}

One of the main purposes of this article is to prove an analogue of Theorem \ref{Pontryagin Duality Theorem} in the case of compact DVRs. To do this, we fix the following notation: $R$ is a compact DVR, with a distinguished uniformizing element $\pi$. We fix $K$ to be the fraction field of $R$. The absolute value on $R$ extends to $K$, endowing $K$ with the structure of a (complete, compact) topological field. We fix $T$ to be the $R$-module quotient $K/R$ endowed with quotient topology, a local analogue of $\RR/\ZZ$.

\begin{theorem}\label{Local Duality Theorem} Let $M$ be a locally compact $R$-module. Define $\widehat{M}$ to be the space of continuous $R$-module homomorphisms $M\xrightarrow{} T$ endowed with the compact-open topology. Then, $\widehat{M}$ is locally compact and the natural map

\begin{align*}
M&\xrightarrow{} \widehat{\widehat{M}}\\
m&\mapsto \left(\varphi \mapsto \varphi(m)\right)
\end{align*}

is an isomorphism of $R$-modules and a homeomorphism of topological spaces.
\end{theorem}

Other results obtained about locally compact $R$-modules reveal a distinctly different structure to that of abelian groups. In particular, we have the following:

\begin{theorem}\label{Topology Forcing} Modules over compact DVRs have at most one topology that make them locally compact topological modules.
\end{theorem}

This statement fails immediately for abelian groups, where any non-discrete locally compact module could also be given the discrete topology. An example of this theorem in action is as follows: Let $K/\QQ_p$ be a finite extension, and $R=\mathcal{O}_K$ be its ring of integers. Then $R$ is a compact DVR, and as such has a natural metric space topology. Additionally, the inclusion $\ZZ_p\hookrightarrow R$ endows $R$ with the structure of a topological $\ZZ_p$-module. Every finitely generated torsion-free module over a DVR is free, and hence we have a $\ZZ_p$-module isomorphism $R\cong \ZZ_p^e$ for some $e\geq 1$. The question is whether or not the product topology on $r$ copies of $\ZZ_p$ and the metric space topology on $R$ agree. Seeing as they are both locally compact, Theorem \ref{Topology Forcing} immediately tells us they are the same. This could also be derived through a longer low-tech direct computation. This method of decomposing into components and applying the product topology is very powerful for all finitely generated modules. In particular, we have the following:

\begin{theorem}\label{fin gen uniqueness} For each finitely generated $R$-module $M$, there is a unique topology on $M$ turning it into a topological module. If $M\cong R$ this topology is obtained endowing $R$ with its metric space topology. If $M\cong R/\mm^n$ for some $n\geq 1$, where $\mm$ is the maximal ideal of $R$, then this topology is discrete. If $M\times N$ is the product of non-trivial modules, then this topology is equal to the product topology on $M$ and $N$.
\end{theorem}

For more general modules, one cannot eliminate the possibility of there being multiple topological module structures on a given algebraic base. However, a canonical choice still distinguishes itself:

\begin{theorem}\label{Equivalences theorem} Given an $R$-module $M$, the following topologies are all well defined and equivalent:

\begin{enumerate}
\item The direct limit topology arising from the equality $M=\varinjlim_{S\leq M} S$, where $S$ runs over finitely generated submodules of $M$ , each given the unique topological module structure of Theorem \ref{fin gen uniqueness}.
\item The unique finest topological module structure on $M$, i.e., the unique structure such that every set open in any topological module structure on $M$ is open in this one.
\item The unique topological module structure on $M$ for which the double dual map $M\xrightarrow{} M^{\wedge\wedge}$ is injective and continuous.
\end{enumerate}

We refer to this topology as the canonical topology on $M$, or refer to $M$ as canonical.
\end{theorem}

With the above results in mind, it is clear that every locally compact topological module has the canonical topology: Their double dual maps are all injective and continuous by Theorem \ref{Local Duality Theorem}, and hence canonical by point (3) of Theorem \ref{Equivalences theorem}. Thus, the statement that there is at most one locally compact topology in Theorem \ref{Topology Forcing} is a simple corollary of the fact that the canonical topology is well defined and unique.

In general the double dual map on canonical modules will not be surjective, with failures as simple as $M=\bigoplus_{n=0}^{\infty} R$, the direct sum of infinitely many copies of $R$. The computation of this fact is not difficult, with the only apparent difficulty coming from the problem of determining which maps in the dual are continuous. This is no problem at all, however, because of the following theorem:

\begin{theorem}\label{fully faithful} Every homomorphism between canonical $R$-modules is continuous. In particular, the functor defined by sending an $R$-module to itself equipped with canonical topology is fully faithful.
\end{theorem}

This is useful, since it allows us to ignore topology in computations with canonical modules and use purely algebraic methods. An analogy can be drawn to the case of abelian groups, where one has a similar embedding sending any group to itself equipped with discrete topology.

While the results of Theorems  \ref{Topology Forcing}, \ref{fin gen uniqueness}, \ref{Equivalences theorem}, and \ref{fully faithful} appear to be new to literature, Theorem \ref{Local Duality Theorem} is referenced implicitly by Flood in \cite{flood1979pontryagin}. Namely, the following setup is considered. We let $A$ be a locally compact topological ring, and define $\TT_A$ to be the space of continuous $\ZZ$-module homomorphisms from $A$ to $\RR/\ZZ$, equipped with the $A$-module structure $(a\cdot \varphi)(x)=\varphi(a\cdot x)$. We define the dual $M^{\wedge}$ of a topological module $M$ to be the space of continuous $A$-module homomorphisms $M\xrightarrow{}\TT_A$ endowed with the compact-open topology. We have the following:

\begin{theorem}\label{simplified adjointness} Let $M$ be an $A$-module. Denote by $M^{\wedge}_{(\ZZ)}$ the $\ZZ$-module dual of morphisms into $\RR/\ZZ$, and by $M^{\wedge}_{(A)}$ the $A$-module dual of morphisms into $\TT_A$. The map

\begin{align*}
M^{\wedge}_{(A)}&\xrightarrow{}M^{\wedge}_{(\ZZ)}\\
\Phi &\mapsto \left(m\mapsto \Phi(m)(1)\right)
\end{align*}

is an isomorphism and homeomorphism.
\end{theorem}

Thus, taking $A$-duals and $\ZZ$-duals is the same, and hence Pontryagin duality as stated in Theorem \ref{Pontryagin Duality Theorem} easily allows one to deduce that taking $A$-duals induces an isomorphism on double duals. While no special mention is made to the case that $A$ is a compact DVR in either \cite{flood1979pontryagin} or earlier works of this sort such as \cite{levin1973locally}, the duality stated in the above paragraph can be used to derive Theorem \ref{Local Duality Theorem} via the following equivalence:

\begin{theorem}\label{Flood isomorphism theorem} There is a non-canonical $R$-module isomorphism and homeomorphism $\TT_R\cong T$.
\end{theorem}

Thus, $T$ and $\TT_R$ yield the same duality theory, and so $\TT_R$ inducing isomorphism/homeomorphisms with its double duals guarantees the same for $T$. To elaborate on the non-canonicalness of the isomorphism $\TT_R\cong T$, in general there is no clear way to write down any isomorphism between the two. To construct one, we first appeal to the classification theorem of complete DVRs given in \cite{serre2013local}, which states that $R$ is either isomorphic to the completed polynomial ring $\FFqx$ over a finite field or isomorphic to the ring of integers $\mathcal{O}_K$ of a finite extension $K/\QQ_p$. In the case of polynomial rings, isomorphisms can be written down explicitly, though still not canonically. In the case of rings of integers of finite extensions the proof is non-constructive.

Out of all the rings covered by \cite{flood1979pontryagin}, some global rings also manifest particularly pleasant duality theories:

\begin{theorem}\label{Global Flood isomorphism} Let $A$ be a non-field domain, and $|\cdot|$ an archimedean absolute value on $A$. Then $A$ is complete and locally compact with respect to its topology if and only if it is isomorphic to either $\ZZ$ or $\ZZ\left[\delta\right]$, $\delta=\left(b+\sqrt{b^2+4a}\right)/2$ where $a$ is a negative integer and $b$ is an integer satisfying $\left|b\right|<2\sqrt{-a}$. In the case that $A=\ZZ[\delta]$, there is a canonical isomorphism and homeomorphism

\begin{align*}
\TT_A&\xrightarrow{\sim} \CC/A.\\
\varphi&\mapsto \varphi(1)+\varphi(\delta)\delta
\end{align*}

Thus,  in all locally compact complete cases the double dual map $M\xrightarrow{}M^{\wedge\wedge}$ on $A$-modules $M$ is an isomorphism and homeomorphism, where $\widehat{M}$ is defined to be the space of continuous $A$-module homomorphisms into $\overline{F}/A$, with $\overline{F}$ being the fraction field of $A$ completed with respect to its topology.
\end{theorem}

To demonstrate the above theory, we work the case of $\ZZ_p$-modules is detail. Suppose we are given a $\ZZ_p$-module $M$. These often times appear in Iwasawa theory, where $M$ will be some Galois group or Selmer group. In this context the dual $M^{\wedge}$ consisting of all morphisms into $\QQ_p/\ZZ_p$ is extremely natural, and is often refered to as the Pontryagin dual of $M$. We have shown that this name is warranted, since our results imply that when equipped with the canonical topology the natural map

$$\Hom_{\ZZ}^{\mathrm{cont}}(M,\RR/\ZZ)\xrightarrow{} \Hom_{\ZZ_p}(M,\QQ_p/\ZZ_p)$$

is an isomorphism and homeomorphism. The above map is defined by taking a morphism $\varphi: M\xrightarrow{} \RR/\ZZ$, and noting that for any $m\in M$ the condition $p^n\cdot m\to 0$ implies that $p^n\varphi(m)\to 0$ and hence $m$ is a rational number with $p$-power denominator, and hence can naturally be identified with an element of $\QQ_p/\ZZ_p$. Thus, the codomain of $\varphi$ can be taken to be $\QQ_p/\ZZ_p$ and the above map is defined to be the identity.

A natural question one might ask is the follows: Given the $\ZZ_p$-module $M$, when is the algebraic double dual map

$$M\xrightarrow{} \Hom_{\ZZ_p}\left(\Hom_{\ZZ_p}(M,\QQ_p/\ZZ_p),\QQ_p/\ZZ_p\right)$$

an isomorphism? Our paper gives a very testable condition: It will be an isomorphism whenever $M$ has any locally compact topological module structure. This is because as soon as $M$ has some locally compact topological module structure, we know that this structure must be canonical and hence the topological dual $M^{\wedge}$ consists of all $\ZZ_p$-module morphisms $M\xrightarrow{} \QQ_p/\ZZ_p$. The topological double dual map is an isomorphism for $M$ by Theorem \ref{Local Duality Theorem}, and using a bit more control (i.e., ``dual of canonical is canonical") we immediately get that the algebraic double dual map is an isomorphism.

This condition is simple to check in many scenerios. When $M$ is a Galois group it naturally comes equipped with a profinite topology which will be compact. When $M$ is torsion the canonical topology is clearly discrete. The condition can be stated in an if and only if manner: The algebraic double dual map on a $\ZZ_p$-module will be an isomorphism if and only if the canonical topology (or any topological module structure for that matter) is ``nuclear" in the sense of \cite{banaszczyk2006additive}.

The theory of compact DVRs is far simpler than that of $\ZZ$ in the context of duality for a few reasons. Firstly $R$ is compact, not just locally compact, which is very useful since compact metric spaces enjoy many nice properties. Secondly, $K$ is already complete so there is no analogue of the discrepancy between $\QQ$ and $\RR$. Lastly, the canonical topology on $T$ is discrete (this is Lemma \ref{T discrete topology}), which makes the consideration of continuous morphisms into $T$ far simpler.

Seeing as the canonical topology was defined as the simultaneous presence of three properties, none of which we have shown to be well defined and none of which we have shown to be equivalent, the notation currently has no meaning. For the sake of the rest of this introduction we use the term canonical to refer to the topology defined in point (1) of Theorem \ref{Equivalences theorem}. The proof that this topology is well defined is relatively straightforward, consisting of a large number of explicit computations; it is performed in the beginning of § \ref{The Canonical Topology}.

We now describe the proofs of Theorem \ref{Equivalences theorem}. The (1)-(2) equivalence is straightforward from Theorem \ref{fin gen uniqueness}, by the universal property of the direct limit. Using explicit calculations in the finitely generated case and limiting arguments in the general case, it is also not too difficult to prove that  the double dual map on canonical modules is injective and continuous. The real difficulty comes in showing that if the double dual map is injective and continuous, then the module is canonical.

Choose a topological module $M$, and let $M_0$ denote that same module equipped with the canonical topology. The (1)-(2) equivalence implies that the identity map $M_0\xrightarrow{}M$ is continuous. Taking duals, we get a continuous injection $M^{\wedge}\xhookrightarrow{\varphi} M_0^{\wedge}$. Restricting to the image, we get a continuous bijection $M^{\wedge}\xhookrightarrow{\varphi} \left.M_0^{\wedge}\right|_{\varphi\left(M^{\wedge}\right)}$. It is then proved that the dual of a canonical module is canonical, and that submodules of canonical modules are canonical. In particular, we have that $\left.M_0^{\wedge}\right|_{\varphi\left(M^{\wedge}\right)}$ is canonical. Thus, $M^{\wedge}$ has a continuous isomorphism into the module equal to it algebraically but with canonical topology, and hence the topology on $M^{\wedge}$ is finer than the canonical topology. However, by point (2) of Theorem \ref{Equivalences theorem} the canonical topology on $M^{\wedge}$ is the finest possible, and so we must have that the compact-open topology on $M^{\wedge}$ is canonical

Since duals of canonical modules are canonical, we have now that $M^{\wedge\wedge}$ is canonical. By assumption the double dual map $M\xrightarrow{}M^{\wedge\wedge}$ is injective and continuous. Using the technique above of restricting to the image and squeezing with point (2) of Theorem \ref{Equivalences theorem}, we get that $M$ must be canonical so we are done.

We now describe the structure of this paper:

$\newline$

§ \ref{The Canonical Topology} proves that the canonical topology exists, as well as showing some of its first properties. Namely, Theorem \ref{fin gen uniqueness} is deduced.

§ \ref{The Double Dual Map} computes the duals of modules of interest. It is shown quickly that Theorem \ref{fully faithful} is true, and later it is shown that the double dual map is always injective and continuous for canonical modules

§ \ref{Topological Modules over Compact DVRs} proves relevant results about topological modules over compact DVRs, culminating in proofs of Theorems \ref{Topology Forcing} and \ref{Equivalences theorem}.

§ \ref{Global to Local} proves Theorem \ref{Flood isomorphism theorem}, as well as delivering a proof of the adjointness result from \cite{flood1979pontryagin}. In particular, enough is shown to deduce Theorem \ref{Local Duality Theorem} using Pontryagin duality.

§ \ref{Completing the Proof} deals with the case that $A$ is a ring that is complete and locally compact with respect to an Archimedean absolute value. Namely, a proof of Theorem \ref{Global Flood isomorphism} is given.

$\newline$

Notation:

\begin{itemize}

\item All rings are commutative unital, and all modules are unital.

\item All topological modules are Hausdorff

\item The word compact is reserved for compact Hausdorff. To say simply compact we will say quasi-compact, and similarly for locally compact.

\item We use $\HM_{A}$ (resp. $\LCM_{A}$) to denote the category whose objects are Hausdorff (resp. locally compact) topological modules and whose morphisms are continuous $A$-module homomorphisms.

\item We will use $\left[\cdot,\cdot\right]_A$ to mean the (internal) hom over $\HM_{A}$ whenever it does not cause confusion, i.e., the space of continuous $A$-module morphisms from the target to the source, with premultiplication module structure. When the ring is understood we drop the subscript.

\item We use self double-dual to refer to topological $a$-modules for which the double dual map is an $\HM_{A}$-isomorphism.

\item When use $V_{M}(Z_1...Z_n;U_1...U_n)$ to denote the set of all elements $\varphi\in\widehat{M}$ for which $\varphi(Z_i)\subseteq U_i$, where $M$ is any topological module. If $M$ is understood, we refer simply to $V(Z_1...Z_n;U_1...U_n)$.

\item For brevity, we occasionally abbreviate ``finitely generated" to fin gen when it does not cause confusion
\end{itemize}

\section{The Canonical Topology}\label{The Canonical Topology}

This section seeks to show the well-definiteness of the canonical topology for finitely generated modules over a fixed compact DVR $R$, as well as some of its first properties. Currently the notion has no meaning, so in the interim we will be referring to product topologies on finitely generated modules: That is, those topologies that arise from first choosing an isomorphism $X\cong \prod R/(x)$ between $X$ and some product of principal quotients, and then putting the canonical topologies (discrete when $x\neq 0$, metric when $x=0$) on each component, and then putting the product topology on $X$. We will show that all product topologies on $X$ are equal, and give rise to the canonical topology.

To begin, we prove a few simple lemmas that will be useful throughout the paper:

\begin{lemma}\label{T discrete topology} The quotient topology on  $T$ is discrete.
\end{lemma}
\begin{proof} The preimage of the equivalence class $[t]\in T$ under the projection $K\xrightarrow{} T$ is $t+R$, which is open in $K$ since it can be equivalently described as the open non-archimedean ball $\{x\in K\st |x-t|\leq 1\}$. Thus, by the definition of the quotient topology, each singleton in $T$ must be open so the topology is discrete.
\end{proof}

\begin{lemma}\label{biproduct} Given any two topological modules $M,N$ over a topological ring $A$, the direct product with product topology $M\times N$ is again a topological $A$-module. Moreover, it is a biproduct in the category of topological modules, i.e., it satisfies the universal properties of both the product and coproduct. Additionally, the product of complete modules is complete.
\end{lemma}
\begin{proof} The fact that $M\times N$ is a topological module follows from an unenlightening exercise in computation. Seeing as $M\times N$ is defined to have both the product module and product topological structure, it clearly satisfies the axioms of a categorical product. Hence, we are left with showing that it satisfies the coproduct axioms as well.

To do this, we take as input a topological module $X$ and two continuous morphisms $f_N:N\xrightarrow{} X$, $f_M: M\xrightarrow{} X$. We wish to construct a map $f: N\times M\xrightarrow{} X$ which acts like $f_N,f_M$ when restricted correctly. To do this, we simply consider both $f_N$ and $f_M$ as morphisms $N\times M\xrightarrow{} X$ by first factoring through a restriction $N\times M \twoheadrightarrow{} M,N $. Taking the sum of $f_N$ and $f_M$ we get the desired continuous map.

We now prove that the product of complete modules is complete. Suppose $(z_n)_{n\geq 0}$ is a Cauchy sequence in $M\times N$. First, we write $z_n=(x_n,y_n)$. A Cauchy sequence in $M\times N$ clearly must be Cauchy in each component, and hence $x_n$ and $y_n$ are both Cauchy. By hypothesis they have limit points $x_{\infty}$ and $y_{\infty}$ respectively, so we can set $z_{\infty}=(x_{\infty},y_{\infty})$. This is clearly a limit point for $(z_n)_{n\geq 0}$, so we are done.
\end{proof}

We now can get to proving results about product topologies on finitely generated modules, when we remember that $\pi$ is a distinguished uniformizing element for $R$:

\begin{lemma}\label{compact complete} All product topologies on finitely generated modules induce compact complete topological structures.
\end{lemma}
\begin{proof} This is clearly true when the module under consideration is $R$, since the metric space topology on $R$ makes it a topological ring. When the module is $R/(\pi^n)$ for some $n\geq 1$, addition and inversion are continuous since the space is discrete. To show scalar multiplication $R\times R/(\pi^n)\xrightarrow{} R/(\pi^n)$ is continuous, we compute the preimage of $0$ to be

$$\bigcup_{x\in R/(\pi^n)} (\pi^{k_x})\times\{x\},$$

where $k_x$ is the smallest integer such that $\pi^k x=0$. Since $(\pi^{k_x})$ is open in $R$ and $x$ is open in $R/(\pi^n)$, the union of these sets is open in $R\times R/(\pi^n)$. Translating using addition we get that the preimage of any singleton in $R/(\pi^n)$ is open, and since singletons form a basis for the topology of $R/(\pi^n)$ we are done with the proof that scalar multiplication is continuous. The last condition needed for $R/(\pi^n)$ to be a topological module is for it to be Hausdorff, which is clear.

Both $R$ and $R/(\pi^n)$ are compact complete topological modules, so by Lemma \ref{biproduct} all product topologies on fin gen modules give compact complete topological structures.
\end{proof}

We now work towards our proof that all topological structures on fin gen modules are equivalent, by showing that there is a unique topological structure on $R$. The proof strongly uses the fact that $R$ is a compact metric space. In fact, it already fails for the case of $\ZZ$-modules. As an example, one can take the set of all arithmetic sequences as a basis of $\ZZ$: This topology is non-canonical (i.e., not discrete), metrizable, and is used in \cite{furstenberg2009infinitude} to deduce the infinitude of primes.

\begin{lemma}\label{R top unique} Every topological $R$-module $M$ algebraically isomorphic to $R$ must have the metric space topology.
\end{lemma}
\begin{proof} Let $M\in\HM_R$ be algebraically isomorphic to $R$. We wish to show $M=R$ topologically as well. To begin we note composition

\begin{align*}
R\xrightarrow{}R&\times M\xrightarrow{} M\\
r\mapsto (&r,1) \mapsto r\cdot 1
\end{align*}

is continuous, so $R$ is finer than $M$. It remains to show that $M$ is finer than $R$, or equivalently that every convergent sequence $(x_n)_{n\geq 0}$ in $M$ is also convergent in $R$ and has the same limit. Assume for the sake of contradiction that the sequence did not $R$-converge to $x$. Then since the topology on $R$ is metrizable there exists an $\epsilon>0$ and a subsequence $(y_n)_{n\geq 0}$ of $(x_n)_{n\geq0}$ such that $|y_n-x|>\epsilon$ uniformly. By compactness there is an $R$-convergent subsequence $(z_n)_{n\geq0}$ of $(y_n)_{n\geq0}$. Denoting the limit point by $z_{\infty}$, we see that $(z_n,1)_{n\geq 0}$ converging to $(z_{\infty},1)$ in $R\times M$ implies $z_n\cdot 1=z_n$ converges to $z\cdot 1=z$ in $M$ by the continuity of multiplication. Thus, since $(z_n)_{n\geq 0}$ converges in $M$ to $x$ and $M$ is Hausdorff, we must have that $z=x$. Clearly our definition of $(y_n)_{n\geq0}$ implies it cannot have any subsequences converging to $x$, hence we have a contradiction, so our proof is complete.
\end{proof}

We are now ready to prove Theorem \ref{fin gen uniqueness}:

\begin{proof}[Proof of Theorem \ref{fin gen uniqueness}] This is true for modules of the form $R/(\pi^n)$ since the only Hausdorff topology on a finite space is discrete, and it is true for $R$ by Lemma \ref{R top unique}. The key step is to show that if $X\cong N\times M\in \HM_R$ is a fin gen topological module, then $X=N\times M$ has the product topology with $N=N\times\{0\}$ and $M=\{0\}\times M$ given subspace topologies.

We first elaborate on why the above statement is enough to deduce the proposition. If we choose an algebraic isomorphism $X\xrightarrow{\sim} \prod R/(x)$, then decomposing $X$ one term at a time and using the above statement we find that the topology on $X$ is equal to the product topology arising from this isomorphism. Hence, all product topologies are equivalent, and every fin gen module has the product topology.

To prove this result, we go by induction on the number of generators of $X$. Suppose $X\cong N\times M$, $N,M\neq 0$. By our hypothesis, $N$ and $M$ both have their unique topologies. By Lemma \ref{compact complete} $N$ and $M$ are both complete, and hence they are closed in $X$. It is a standard result that taking quotients by closed subgroups results in Hausdorff spaces (see e.g., \cite{hewitt2012abstract} Theorem 5.21). Hence the quotient topologies on $N=X/(\{0\}\times M)$ and $M=X/(N\times\{0\})$ are topological module structures, and so by the uniqueness of structures on $M$ and $N$ we find that the quotient topologies and subspace topologies agree. In particular the surjections  $X\twoheadrightarrow N,M$ serve as $\HM_R$-sections to the injections $N,M \xhookrightarrow{} X$, and hence the top and bottom rows in the diagram

\[
\begin{tikzcd}
0 \arrow[r] & N \arrow[d, "\sim"] \arrow[r] & X \arrow[d] \arrow[r] & M \arrow[d, "\sim"] \arrow[r] & 0 \\
0 \arrow[r] & N \arrow[r,] & N\times M \arrow[r] & M \ar[r] & 0
\end{tikzcd}
\]

are both split exact. The middle arrow of this diagram comes from the universal property of the product, which is valid by Lemma \ref{biproduct}. Seeing as all these maps are the identity on sets we conclude the diagram is commutative. The short split 5-lemma is the statement in a given category that whenever there is a split exact sequence of this form then the middle arrow is an isomorphism. Theorem 50 of \cite{borceux2005topological} shows that the short split 5-lemma holds in the category of topological modules over any ring, hence $X=N\times  M$ and our proof is complete.

\end{proof}

We thus can now refer to the canonical topology on fin gen modules. For general modules, it is clear that passing to the direct limit $M=\varinjlim_{S\leq M}S$ with $S$ running over fin gen submodules with canonical topology does indeed give a topology. A key point used is that all of the connecting maps $S_0\hookrightarrow S_1$ are continuous, which is true since the subspace topology of $S_0\leq S_1$ always induces a topological module structure (additional/scalar multiplication/inversion are still continuous), hence it must agree with the canonical topology on $S_0$ by Theorem \ref{fin gen uniqueness}. It is still not clearly a topological module structure however, and so we prove this fact after a necessary proposition:

\begin{proposition}\label{Lemma2} If $M$ is in canonical form then it is Hausdorff regular, and all submodules of $M$ are closed and canonical when given subspace topology.
\end{proposition}
\begin{proof} To prove that $M$ is Hausdorff, we appeal to the general theory of sets given the direct limit topology on an ascending sequence of compact Hausdorff subspaces. They are known as $k_{\omega}$-spaces and enjoy many nice properties, one of which being that they are all Hausdorff regular. This was originally proved in the Russian language paper \cite{graev1948free}, and subsequently exposited in a more modern fashion by \cite{franklin1977survey}.

Now, let $N\leq M$ be a submodule. To show $N$ is closed it suffices to show that $N \cap S $ is closed in $S$ for all fin gen submodules $S\leq M$. This is clear, however, since $N\cap S$ is complete by Lemma \ref{compact complete} and thus is closed since it contains all of its limit points.

We now show that $N$ must be canonical, by showing that a set $W\subseteq N$ is closed in the subspace topology if and only if it is closed in the canonical topology. If $W$ is closed in the subspace topology then $W=X\cap N$ for some closed $X\subseteq M$, and hence

$$W\cap S=(X\cap N)\cap S=X\cap S=\left(\mathrm{closed}\right)$$

for all $S\leq N$, and hence $W$ is closed in the canonical topology. Conversely, if $W$ is closed in the canonical topology then we must show that $W=X\cap N$ for some closed set $X\subseteq M$. This is clear, however, since $W$ is closed in $N$ which is closed in $M$, and hence $W$ is closed in $M$, so we can take $X=W$. Thus, we conclude the result.
\end{proof}

\begin{proposition} The canonical topology always induces an $\HM_R$-structure on modules.
\end{proposition}
\begin{proof} Write $M=\varinjlim_{S\leq M} S$ with $S$ running over fin gen submodules. To show that inversion $f: M\xrightarrow{} M$ is continuous, we compute that

\begin{align*}
\left(f\,\, \mathrm{continuous}\right) &\iff \left(f^{-1}(U)\cap S \,\, \mathrm{open}\,\, \forall\, S\leq M,\, U\subseteq M\, \mathrm{open}\right)\\
&\iff \left(\left.f\right|_S: S\xrightarrow{} M \,\, \mathrm{continuous}\,\, \forall S\leq M\right).
\end{align*}

Now, $\left.f\right|_S$ is exactly the inversion map on $S$ and hence since $S$ with subspace topology is canonical by Proposition \ref{Lemma2} the inversion map on $S$ must be continuous by the previously treated case.

We now show that the addition map $f:M\times M\xrightarrow{} M$ is continuous. Choose $U\subseteq M$ open. We wish to show that $f^{-1}(U)\cap S$ is open in $S$ for all $S\leq M\times M$ finitely generated. Denoting $p_0,p_1:M\times M\xrightarrow{}M$ to be the projection maps onto each component, we let $S'$ be the submodule of $M$ generated by $p_0(S)$ and $p_1(S)$. Since $p_0(S)$ and $p_1(S)$ are fin gen, so is $S'$. Restricting the addition map on $M$ to $S'\times S'$, we obtain the addition map on $S'$. By Proposition \ref{Lemma2}, $S'$ is canonical when given subspace topology and hence the addition map is continuous by the previously treated case. Since $S\leq S'\times S'$, the set $f^{-1}(U)\cap \left(S'\times S'\right)$ being open in $S'\times S'$ implies $f^{-1}(U)\cap S$ is open in $S$ and so we are done.

We now show that the scalar multiplication map $f: R\times M\xrightarrow{} M$ is continuous. Again, this reduces to showing that $f^{-1}(U)\cap S$ is open for all $U\subseteq M$ open and all $S\leq R\times M$ finitely generated. Denoting $p: R\times M\xrightarrow{} M$ to be the canonical projection, we let $S'=p(S)$. Restricting the multiplication map to $R\times S'$ we obtain the scalar multiplication map on $S'$. By Proposition \ref{Lemma2}, $S'$ is canonical when given subspace topology and hence the addition map is continuous by the previously treated case. Since $S\leq R\times S'$, the set $f^{-1}(U)\cap \left(R\times S'\right)$ being open in $R\times S'$ implies $f^{-1}(U)\cap S$ is open in $S$ and so we are done.
\end{proof}

As a corollary, we deduce the (1)-(2) equivalence of Theorem \ref{Equivalences theorem}:

\begin{corollary}\label{Lemma3} If $M\in\HM_R$ and $M_0$ is the module algebraically equal to $M$ but given the canonical topology, then the identity map $M_0\xrightarrow{} M$ is continuous (i.e., $M_0$ is finer than $M$).
\end{corollary}
\begin{proof} Letting $S$ run over the fin gen submodules of $M$ with subspace topology, the universal property of the direct limit gives a continuous identity map $\varinjlim_{S\leq M}S \xrightarrow{} M$. By uniqueness each $S$ must have canonical topology, and hence $\varinjlim_{S\leq M}S=M_0$ so we are done.
\end{proof}

\section{The Double Dual Map}\label{The Double Dual Map}

In this section we perform computations on $R$-module morphisms between finitely generated modules. To begin, we show that every morphism between canonical modules is continuous:

\begin{proof}[Proof of Theorem \ref{fully faithful}] Every morphism leaving a quotient $R/(x)$ for $x\neq 0$ is clearly continuous, since $R/(x)$ is discrete. Let $f: R\xrightarrow{} M$ be a morphism between $R$ and a module of canonical form. We show $f$ is continuous. To begin, we factor $f$ through its kernel as

$$f:R\twoheadrightarrow{} R/(x)\hookrightarrow{} M.$$

By Proposition \ref{Lemma2}, subspaces of canonical modules are canonical, so $R/(x)$ is canonical. As a first case, we set $x\neq 0$ so $R/(x)$ is discrete. The preimage of $0$ in $R\twoheadrightarrow{} R/(x)$ is $(x)$. Seeing as $(x)$ is open in $R$ for any choice of $(x)$, we get that $R\twoheadrightarrow{} R/(x)$ is continuous. For a second case, we set $x= 0$ so $R/(x)=R$. The only injective $R$-module morphisms $R\xrightarrow{}R$ are multiplications by nonzero elements, which are clearly all continuous. Thus, since the inclusion $R/(x)\hookrightarrow{} M$ is continuous by the definition of the subspace topology, we get that $f$ is continuous.

We now show that all maps leaving fin gen modules are continuous. We choose $S$ fin gen canonical, $M$ canonical, and a morphism $f: S\xrightarrow{} M$. We write $S$ as a product $\prod R/(x)$. By the coproduct property of Lemma \ref{biproduct}, giving a continuous map $S\xrightarrow{}M$ is the same as giving a continuous map leaving $R/(x)$ for each $R/(x)$. Since all maps leaving $R/(x)$ are continuous by the previous computations, we get that all maps leaving $S$ are thus continuous so we are done.

We now do the case of general modules. Choose a morphism $f:M\xrightarrow{} N$, both modules canonical. The map $f$ is continuous if and only if $f^{-1}(U)$ is open for all $U\subseteq N$ open, which in turn is true if and only if $f^{-1}(U)\cap S$ is open for all $S\leq M$ fin gen. This is clearly equivalent to the restrictions $\left.f \right|_{S}$ being continuous for each fin gen module. Seeing as $\left.f \right|_{S}$ is a morphism leaving a fin gen module, it must be continuous for each $S$. Thus, $f$ is continuous so we are done.
\end{proof}

We will be dealing a lot with the quotients $R/(x)$ of $R$ by nonzero elements, and so we fix the following notation. When $x,y\in R$ are such that $x$ divides $y$, we write $\inf: R/(x)\hookrightarrow R/(y)$ for the injection sending $r\in R/(x)$ to $r\cdot (y/x)$ in $R/(y)$. Additionally, we write $\res: R/(y) \twoheadrightarrow R/(x)$ to denote taking $r\in R/(y)$ to the well-defined coset in $R/(x)$ associated to any $R$-representative.

\begin{proposition}\label{First Full R Lemito} For all nonzero $x\in R$, there exists a canonical $\HM_R$-isomorphism $\widehat{R/(x)}\xrightarrow{i_x} R/(x)$ such that for all $x,y\in R$ with $x$ dividing $y$ we have a commutative diagram

\[
\begin{tikzcd}
\widehat{R/(y)} \arrow[d,"\inf^{\wedge}"]\arrow[r,"i_y"] & R/(y)\arrow[d,"\res"]\\
\widehat{R/(x)}\arrow[r,"i_x"] & R/(x).\\
\end{tikzcd}
\]
\end{proposition}
\begin{proof} Since $R/(x)$ is freely generated by the element $1$ and the condition that $x\cdot 1=0$, we have that $R$-module maps $\varphi: R/(x)\xrightarrow{} T$ are uniquely defined by $\varphi(1)$, and are well defined if and only if $x\cdot \varphi(1)=0$. All of the maps are continuous by Theorem \ref{fully faithful}, and hence we naturally have an isomorphism between $(R/(x))^{\wedge}$ and the $x$-torsion of $T$, sending $\varphi$ to $\varphi(1)$. An element $t\in T$ is $x$-torsion if and only if it can be written in the form $r_t/x$ where $r_t\in R$. Sending $t$ to $r_t$, which is well defined up to adding a multiple of $x$, we thus get an isomorphism between the $x$-torsion of $T$ and $R/(x)$. Hence, collecting, we have an isomorphism between $(R/(x))^{\wedge}$ and $R/(x)$, which we call $i_x$.

We now check that the desired square is commutative. Choose $\varphi\in (R/(y))^{\wedge}$. Mapping through $i_y$, we find that $\varphi$ sends to the unique $a\in R/(y)$ such that $\varphi(1)=a/y$. Mapping down, we send to the coset $a$ in $R/(x)$. Around the other direction, $\varphi$ maps to the $\tilde{\varphi}\in (R/(x))^{\wedge}$ defined by $\tilde{\varphi}(r)=\varphi((y/x)\cdot r)$. Thus,

$$\tilde{\varphi}(1)=\varphi\left((y/x)\cdot 1\right)=(y/x)\cdot \varphi(1)=(y/x)\cdot (a/y)=a/x,$$

and hence $\tilde{\varphi}$ maps down to $a$. We see that the result going around both directions is the same, so we are done.

\end{proof}

\begin{proposition}\label{Second Full R Lemito} The double dual map $R/(x)\xrightarrow{} \widehat{\widehat{R/(x)}}$ is an $\HM_R$-isomorphism.
\end{proposition}
\begin{proof} The kernel of this map consists of those $r$ such that $\varphi(r)=0$ for all $\varphi\in (R/(x))^{\wedge}$. In particular, such elements must be annihilated by the map $R/(x)\xrightarrow{} T$ which has $\varphi(1)=1/x$. Clearly $\varphi(r)=r/x$ is $0$ in $T$ if and only if $r=0$ in $R/(x)$, and hence we get that the double dual map is injective.

Thus, $R/(x)\xrightarrow{} (R/(x))^{\wedge\wedge}$ is thus an injection of $R/(x)$ into a module with which it is algebraically isomorphic, by Proposition \ref{First Full R Lemito}. In particular, both modules have the same number of elements and hence injectivity implies surjectivity by counting. Thus, the map is an isomorphism. It is also a homeomorphism of topological spaces since the dual of a Hausdorff space is Hausdorff (Proposition \ref{LCM duality closed}) so $(R/(x))^{\wedge\wedge}$ is finite Hausdorff so discrete. Thus, our proof is complete
\end{proof}

\begin{proposition}\label{Rhat isomorphism} The map $\widehat{R}\xrightarrow{} T$ sending $\varphi$ to $\varphi(1)$ is an $\HM_R$-isomorphism.
\end{proposition}
\begin{proof} By Theorem \ref{fully faithful} all of the morphism defined are continuous, and so it is clear that this map is a bijection. What is left to show is that the induced topology on $R^{\wedge}$ agrees with that of $T$, i.e., that it is continuous (by Lemma \ref{T discrete topology}). This is obvious, however, since the open sets $V(\left\{1\right\},\varphi(1))$ single out each $\varphi$.
\end{proof}

\begin{proposition}\label{That isomorphism}The map $\widehat{T}\xrightarrow{} R$ sending $\varphi$ to $\varphi(1)$ is an $\HM_R$-isomorphism.
\end{proposition}
\begin{proof} To begin, we observe that there is an $R$-module isomorphism

$$\varinjlim_{x(\neq0)\in R}R/(x)\xrightarrow{\sim} T$$

sending $a_x\in R/(x)$ to $a_x/x$, where the direct limit structure is induced by the $\inf$ maps. Since $R/(x)$ and $T$ are both discrete, this is also a homeomorphism. Next, we get a canonical bijection

$$\widehat{T}\xrightarrow{\sim}\varprojlim_{x(\neq0)\in R}\left[R/(x),T\right].$$

using the fact that the representable functor $\cdot \mapsto [\cdot,T]$ turns limits into colimits. By Proposition \ref{First Full R Lemito}, we have isomorphisms $i_x:\left[R/(x),T\right]\xrightarrow{} R/(x)$ compatible with the inverse limit structure and hence

$$\varprojlim_{x(\neq0)\in R}\left[R/(x),T\right]=\varprojlim_{x(\neq0)\in R}R/(x)=R.$$

Thus $T^{\wedge}$ is algebraically equal to $R$, and in particular by Theorem \ref{fin gen uniqueness} we get that the compact-open topology on $T^{\wedge}$ is canonical and hence the bijection $T^{\wedge}\xrightarrow{\sim}\varprojlim\left[R/(x),T\right]$ is an $\HM_R$-isomorphism.

We now show that this isomorphism agrees with the one in the statement of the proposition. Given $\varphi\in T^{\wedge}$, its image in $\varprojlim_{x(\neq0)\in R}\left[R/(x),T\right]$ is the element $(\varphi_x)_{x(\neq0)\in R}$ where

$$\varphi_x(r)=\varphi(r/x).$$

Following the definition of $i_x$, each $\varphi_x$ sends to $\varphi_x(1)=\varphi(1/x)$, which then is written as $\varphi(1)/x$, which is sent to $\left.\varphi(1)\right|_{R/(x)}$. Thus, the image of $\varphi$ in $\varprojlim_{x(\neq0)\in R}R/(x)$ is

$$\left(\left.\varphi(1)\right|_{R/(x)}\right)_{x(\neq 0)\in R}.$$

This element clearly corresponds to $\varphi(1)\in R$, and hence we are done.

\end{proof}

\begin{corollary}\label{Final R Lemma} The double dual maps $R\xrightarrow{} \widehat{\widehat{R}}$ and $T\xrightarrow{} \widehat{\widehat{T}}$ are $\HM_R$-isomorphisms.
\end{corollary}
\begin{proof} We prove the double dual map $d:R\xrightarrow{} R^{\wedge\wedge}$ is an isomorphism, and omit the proof of the second statement since it is completely analogous. Let $i: R^{\wedge\wedge} \xrightarrow{} R$ denote the composition

$$\widehat{\widehat{R}}\xrightarrow{i_1} \widehat{T}\xrightarrow{i_2}R,$$

where $i_1$ is induced from the isomorphism $R^{\wedge}\xrightarrow{}{T}$ of Proposition \ref{Rhat isomorphism} and $i_2$ is the isomorphism $T^{\wedge}\xrightarrow{}{R}$ of Proposition \ref{That isomorphism}. We compute for any $x\in R$

\begin{align*}
(i\circ d)(x)&=i_2\left(i_1\left(\left(\varphi\mapsto \varphi(x)\right)\right)\right)\\
&=i_2\left(\left(t\mapsto x\cdot t\right)\right)\\
&=x.
\end{align*}

Hence $d$ postcomposed with an isomorphism is an isomorphism, so $d$ itself is an isomorphism and we are done.
\end{proof}

\begin{proposition}\label{FinGenMod case}Let $M$ be a finitely generated $R$-module. Then the double dual map $M\xrightarrow{} M^{\wedge\wedge}$ is an $\HM_R$-isomorphism.
\end{proposition}
\begin{proof} Since $M$ is the direct sum of copies of $R$ and modules of the form $R/(x)$ for nonzero $x\in R$, using the coproduct property of the direct product in Lemma \ref{biproduct}, we conclude from Proposition \ref{Second Full R Lemito} and Corollary \ref{Final R Lemma}.
\end{proof}

We now prove a technical lemma that may appear esoteric, but it plays a key role in our proof that the double dual map is injective and continuous. This is because in any suitably nice category the double dual map $\varinjlim A\xrightarrow{}[[\varinjlim A,B],B]$ on a direct limit can be factored as

$$\varinjlim A\xrightarrow{i_1} \varinjlim [[A,B],B]\xrightarrow{i_2} [\varprojlim[A,B],B]\xrightarrow{i_3}[[\varinjlim A,B],B].$$

The map $i_1$ comes from the compostum of all the double dual maps $A\xrightarrow{}[[A,B],B]$. One can easily check that the double dual map is functorial, and makes the appropriate diagrams commute so that $i_1$ is indeed a morphism of direct limits.

The map $i_2$ comes from sending a morphism $f: [A,B]\xrightarrow{} B$ to the composition

$$\varprojlim [A,B] \xrightarrow{} [A,B] \xrightarrow{f} B,$$

where the morphism $\varprojlim [A,B] \xrightarrow{} [A,B]$ is projected onto the component.

The map $i_3$ is the dual of the morphism $[\varinjlim A,B]\xrightarrow{}\varprojlim[A,B]$ sending $f: \varinjlim A \xrightarrow{} B$ to the system $(\left.f\right|_{A})$ of restrictions.

The problem of analyzing the double-dual map in nice circumstances is thus reduced to analyzing these arrows one at a time. The next lemma serves to describe how the map $\varinjlim [[A,B],B]\xrightarrow{} [\varprojlim[A,B],B]$ behaves in our case.

\begin{lemma}\label{Lemma1} For any $R$-module $M$, we let $S\leq M$ run over finitely generated submodules of $M$ given canonical topology. The natural map

$$\varinjlim_{S\leq M}\left[[S,T],T\right]\xrightarrow{} \left[\varprojlim_{S\leq M}[S,T],T\right]$$

is injective and continuous.
\end{lemma}
\begin{proof} We denote the above map as $i$. For all fin gen submodules $S'$, we denote the projection $\varprojlim[S,T]\xrightarrow{} [S',T]$ by $p_{S'}$. Choose two fin gen submodules $S_0,S_1$, and maps $\Phi_0: [S_0,T]\xrightarrow{} T$, $\Phi_1: [S_1,T]\xrightarrow{} T$ with the same image under $i$. This means that the square

\[
\begin{tikzcd}
\varprojlim_{S\leq M}\left[S,T\right]\arrow[r,"p_{S_0}"]\arrow[d,"p_{S_1}"]\arrow[dr,"\Phi",dotted]& \left[S_0,T\right]\arrow[d,"\Phi_0"] \\
\left[S_1,T\right]\arrow[r,"\Phi_1"]& T\\
\end{tikzcd}
\]

commutes, where $\Phi$ is the composition defined by going along the square in either direction. We now show that $\Phi$ factors through $[S_0\cap S_1,T]$. Choose $\varphi,\varphi': (S_0+S_1)\xrightarrow{}T$ with the same restriction to $S_0\cap S_1$, where $S_0+S_1$ denotes the $R$-module generated by $\alpha$ and $\beta$. We wish to show $\Phi(\varphi)=\Phi(\varphi')$. To do this, we define $\varphi'':(S_0+S_1)\xrightarrow{}T$ by the formula

$$\varphi''(r_0+r_1)=\varphi(r_0)+\varphi'(r_1)$$

where $r_0\in S_0$ and $r_1\in S_1$. The expression $r_0+r_1$ of an element in $S_0+S_1$ is only unique up to moving across an element of $S_0\cap S_1$. This still leads to a well defined value for $\varphi''$, however, since $\varphi$ and $\varphi'$ are homomorphisms that agree on that intersection. The map $\varphi''$ is continuous, since its domain is fin gen, hence it is canonical by Theorem \ref{fin gen uniqueness}, and all maps leaving canonical modules are continuous by Theorem \ref{fully faithful}

Thus, since $\varphi''$ and $\varphi$ agree on $S_0$, we have that $\Phi(\varphi'')=\Phi(\varphi')$. Since $\varphi''$ and $\varphi'$ agree on $S_1$, we have that $\Phi(\varphi'')=\Phi(\varphi)$. Combining, we have that $\Phi(\varphi)=\Phi(\varphi')$ as desired. Thus, the action of $\Phi$ depends only on the behavior in $S_0\cap S_1$. In particular, we can define a map $\Phi_2: [S_0\cap S_1,T]\xrightarrow{T}$ by sending a map $\varphi: (S_0\cap S_1)\xrightarrow{}T$ to $\Phi(\tilde{\varphi})$, where $\tilde{\varphi}$ is any lift of $\varphi$ to $\varprojlim_{S\leq M}[S,T]$. The fact that we can lift $\varphi$ to $\varprojlim_{S\leq M}[S,T]$ is a special case of Baer's criterion, which in this case states that morphisms any $R$-modules $X$ with $\pi X=X$ can be lifted to any enlarged domain. Additionally, the domain of $\Phi_2$ is the dual of a finite product of copies of $R/(x)$ for  $x\neq 0$ and $R$, which is isomorphic to a finite product of copies of $R/(x)$ for $x\neq0$ and $R$, and hence is discrete. Any map leaving a discere domain is continuous, so $\Phi_2$ is continuous. Collecting, we get that the following diagram commutes:

\[
\begin{tikzcd}
& \varprojlim_{S\leq M}\left[S,T\right]\arrow[dl,"p_{S_0}"] \arrow[d,"p_{S_0\cap S_1}"]\arrow[dr,"p_{S_1}"]& \\
\left[S_0,T\right]\arrow[r]\arrow[dr,"\Phi_0"] & \left[S_0\cap S_1, T\right]\arrow[d,"\Phi_2"] & \left[S_1, T\right]\arrow[dl,"\Phi_1"]\arrow[l]\\
& T & \\
\end{tikzcd}
\]
By the definition of the direct limit's equivalence relation, this means that $\Phi_{0}=\Phi_{1}=\Phi_{2}$ in $\varinjlim \left[[S,T],T\right]$. Hence, since $\Phi_0$ and $\Phi_1$ were chosen arbitrarily, we are done.

To demonstrate continuity, it is sufficient to prove that $i^{-1}(V_{\varprojlim[S,T]}(Z,U))$ is open for compact $Z\subseteq \varprojlim_{\alpha}[M_\alpha,T]$ and open $U\subseteq T$. Namely, for all $S_0$ we compute

\begin{align*}
&i^{-1}\left(V_{\varprojlim[S,T]}(Z,U)\right)\cap \left[[S_0,T],T\right]\\
&=\left\{\Phi: [S,T]\xrightarrow{}T \st \Phi (p_{S}(Z))\subseteq U,\,\,S\leq M\right\}\cap \left[[S_0,T],T\right]\\
&=\left\{\Phi: [S_0,T]\xrightarrow{}T \st \Phi (p_{S_0}(Z))\subseteq U\right\}\\
&=V_{[S_0,T]}(p_{S_0}(Z),U)
\end{align*}

to be open in $[[S_0,T],T]$ since $p_{S_0}(K)$ is compact, as the image of compact is compact. By the definition of the direct limit topology, we have finished proving the continuity of the map.
\end{proof}

\begin{lemma}\label{Lemma4} The double dual map on every canonical module is injective and continuous.
\end{lemma}
\begin{proof} Write $M=\varinjlim_{S\leq M} S$, with $S$ running over fin gen submodules as usual. We find that the double dual map $M\xrightarrow{} M^{\wedge\wedge}$ can be factored as the composition of canonical maps

$$M= \varinjlim_{S\leq M} S \xrightarrow{i_1} \varinjlim_{S\leq M} [[S,T],T] \xrightarrow{i_2}  [\varprojlim_{S\leq M}[S,T],T]\xrightarrow{i_3}  [[\varinjlim_{S\leq M}S,T],T]=\widehat{\widehat{M}}.$$

The map $i_1$ is an $\HM_R$-isomorphism by Proposition \ref{FinGenMod case}. The map $i_2$ is continuous injection by Lemma \ref{Lemma1}. The map $i_3$ is an $\HM_R$-isomorphism by Proposition \ref{Lemma5}. The composition of continuous injections is a continuous injection, and so we are done.
\end{proof}

\section{Topological Modules over Compact DVRs}\label{Topological Modules over Compact DVRs}

In this section we prove relevant results pertaining to topological $R$-modules. We begin with a standard proposition about the compact-open topology, which is useful since it realizes duality as a contravariant functor $\HM_R\xrightarrow{}\HM_R$:

\begin{proposition}\label{Dual continuity} For any $\HM_{R}$-morphism $\varphi: M\xrightarrow{}N$, the precomposition map $\widehat{N}\xrightarrow{\widehat{\varphi}}\widehat{M}$ is continuous. By precomposition, we mean that  $\phi:N\xrightarrow{} T$ maps to $\phi\circ \varphi:M\xrightarrow{}T$.
\end{proposition}
\begin{proof} Given any open set $V_{M}(Z;U)\subseteq M^{\wedge}$, we find that

\begin{align*}
\widehat{\varphi}^{-1}\left(V_M(Z,U)\right)&=\left\{\phi\in \widehat{N} \st (\phi\circ \varphi)(Z)\subseteq U\right\}\\
&=\left\{\phi\in \widehat{N} \st \phi\left(\varphi(Z)\right)\subseteq U\right\}\\
&=V_N(\varphi(Z),U).
\end{align*}

Since the image of a compact set is compact, we are done.
\end{proof}

The following argument is taken from the proof of Proposition 1.1 in \cite{anh1981duality}. In this paper, Anh says that ``It is well known that injectivity plays a very important role in duality theory". Despite this statement being in reference to Morita duality, it is equally apt here. We repeatedly use the ability to lift continuous morphisms, making the below proposition a core part of the proof.

\begin{proposition}\label{T almost injective} Let $R$ be a compact DVR, let $M\in \HM_{R}$ be an object, and let $N$ be a submodule of $M$. The precomposition map $\widehat{M}\xrightarrow{}\widehat{N}$ is surjective, i.e.,every $\HM_R$-morphism $N\xrightarrow{}T$ can be lifted to a continuous map $M\xrightarrow{}T$.
\end{proposition}
\begin{proof} Consider $\varphi\in N^{\wedge}$. We wish to show that there is a morphism in $M^{\wedge}$ whose restriction to $N$ is equal to $\varphi$. Since $\ker \varphi$ is the preimage of $\{0\}$, $T$ being discrete (Lemma \ref{T discrete topology}) implies that $\ker \varphi$ is open. By the definition of the subspace topology this means that $\ker \varphi = U \cap N$ where $U$ is an open subset of $M$. Since $\ker \varphi$ is a submodule we have that $U\cap N=\tilde{U}\cap N$ where $\tilde{U}$ is the submodule generated by $U$. Since the $\tilde{U}$ contains an open neighborhood of one of its points it must contain an open neighborhood of all its points by shifting and using the continuity of addition, hence $\tilde{U}$ is open.

We now define a morphism

\begin{align*}
\tilde{\varphi}: N+\tilde{U} &\xrightarrow{} T.\\
n+u &\mapsto \varphi(n)
\end{align*}

The choice of decomposition of an element into the form $n+u$ is unique up to moving across an element of $\tilde{U}\cap N$, and seeing as $\varphi$ acts by zero on this set the map is indeed well defined. The preimage of $t\in T$ is equal to $\varphi^{-1}(t)+\tilde{U}$, which is open since the sum of open sets is open, and hence $\tilde{\varphi}$ is continuous.

It is clear that $\pi T=T$. Baer's criterion (c.f, \cite{baer1940abelian}) states that this is sufficient for $T$ to be injective in the category of $R$-modules, meaning we can algebraically lift $\tilde{\varphi}$ all the way to $M$. This lift is continuous on an open neighborhood of $0$, namely $\tilde{U}$. By continuity of addition, this means that $\tilde{\varphi}$ is continuous on $x+\tilde{U}$ for all $x\in M$, and hence it is continuous on a neighborhood of every point so we are done.
\end{proof}

\begin{corollary}\label{T cogenerator} Over $\HM_R$, $\widehat{M}=0$ if and only if $M=0$.
\end{corollary}
\begin{proof} Choose a nonzero $x\in M$. The submodule $R\cdot x$ must be algebraically isomorphic either $R/(\pi^n)$ or $R$, with $x$ mapping to $1$, and by Theorem \ref{fin gen uniqueness} the topology must be canonical. Either way, the map $R\cdot x \xrightarrow{} T$ distinguished by the property that $x$ sends to $\pi^{-1}$ is well defined and continuous. Thus, we get that for any nonzero $x\in M$ there exists a continuous map $R\cdot x\xrightarrow{} T$ which does not act by $0$ on $x$. Lifting with Proposition \ref{T almost injective} we are done.
\end{proof}

It is an important part of Pontryagin duality theory that the dual of a locally compact abelian group is again locally compact, though that fact is very nontrivial to prove. It requires a proper development of the theory of measures and integration. The proof below is partially due to Eric Wofsey, through collaboration on the online Q$\&$A site Mathematics Stack Exchange, and is much more elementary in nature.

\begin{proposition}\label{LCM duality closed} The categories $\HM_R$ and $\LCM_R$ are closed under taking duals.
\end{proposition}
\begin{proof} To show that $\HM_R$ is closed under taking duals, we choose $M\in \HM_R$ and any unequal $\varphi_0,\varphi_1\in M^{\wedge}$. Clearly, we can find $m\in M$ such that $\varphi_0(m)\neq \varphi_1(m)$. By Lemma \ref{T discrete topology} we have that $T$ is discrete so $\{\varphi_0(m)\},\{\varphi_1(m)\}$ are open. The sets $V(\{m\},\left\{\varphi_0(m)\right\})$ and $V(\{m\},\left\{\varphi_1(m)\right\})$ are thus disjoint open neighborhoods separating $\varphi_0$ and $\varphi_1$, so $M^{\wedge}$ is Hausdorff.

Now, we show that $\LCM_R$ is closed as well. Choose $M\in\LCM_R$.  Given any $\varphi\in M^{\wedge}$, the discreteness of $T$ (Lemma \ref{T discrete topology}) tells us that $\ker\varphi$ is clopen and hence by local compactness of $M$ we have an inclusion of subsets $U\leq Z\leq \ker\varphi$ with $U$ nonzero open and $Z$ compact. Dually, this gives an inclusion of subsets of $M^{\wedge}$

$$V(\ker\varphi,\{0\})\leq V(Z,\{0\})\leq V(U,\{0\}).$$

Clearly $\varphi\in V(\ker\varphi,\{0\})$, and so $V(Z,\{0\})$ is an open neighborhood of $f$. It remains to be shown that $V(U,\{0\})$ is compact. To do this, we note that $V(U,\{0\})=V(\tilde{U},\{0\})$ where $\tilde{U}$ is the submodule of $M$ generated by $U$. $\tilde{U}$ is open, because addition is continuous and so the open neighborhood of $0$ translates to open neighborhoods of every point of $\tilde{U}$.

Now, $V(\tilde{U},0)$ can be naturally identified with $(M/\tilde{U})^{\wedge}$. The topology on $M/\tilde{U}$ is discrete, and so the compact-open topology on $(M/\tilde{U})^{\wedge}$ is equal to the product topology. Additionally, $M/\tilde{U}$ is a torsion module since otherwise the sequence $(\pi^n x)_{n=1}^{\infty}$ would be non-constant convergent. Hence, $(M/\tilde{U})^{\wedge}$ is a subspace of the product $\prod_{x\in M/\tilde{U}}\pi^{-n_x}R/R$ where $n_x$ are chosen such that $\pi^{n_x}x=0$. Each $\pi^{-n_x}R/R\cong R/(\pi^{n_x})$ is finite, so compact, and thus by Tychonoff's theorem $\prod_{x\in M/\tilde{U}}\pi^{-n_x}R/R$ is compact as well. The subset consisting of homomorphisms $M/\tilde{U}\xrightarrow{} T$ is clearly closed, and hence $(M/\tilde{U})^{\wedge}$ is compact so we are done.
\end{proof}

The following lemma condenses the ``restrict to image and squeeze" procedure described in the introduction, which will be used keyly in the proofs of Theorem \ref{Equivalences theorem}.

\begin{lemma}\label{continuous injections} If $N\xhookrightarrow{}M$ is an injection in $\HM_R$ and $M$ is canonical, then so is $N$.
\end{lemma}
\begin{proof} Restricting the map we get a continuous bijection from $N$ to its image, which is canonical by Proposition \ref{Lemma2}. Continuity means that $N$ is finer than its image, and hence the topology on $N$ is finer than the canonical topology. By Corollary \ref{Lemma3} the canonical topology is the finest possible and hence $N$ must be canonical as well.
\end{proof}

\begin{lemma}\label{Lemma5} The dual of a module $M$ in canonical form is also in canonical form, and is $\HM_R$-isomorphic to $\varprojlim_{S\leq M}[S,T]$.
\end{lemma}
\begin{proof} We write $M=\varinjlim_{S\leq M} S$, where for all limits we will be running over fin gen submodules. Given any fin gen submodule $L\leq M^{\wedge}$, we define $[S,T]_{L}$ to be the submodule of $[S,T]$ obtained by projecting $L$ onto $[S,T]$. There is a canonical set injection $\varinjlim_{L\leq M^{\wedge}} [S,T]_{L}\xrightarrow{}[S,T]$. To show this map is surjective, we choose any $\varphi\in[S,T]$. Since $M^{\wedge}\xrightarrow{} [S,T]$ is surjective (Proposition \ref{T almost injective}) we find $\tilde{\varphi}\in M^{\wedge}$ which restricts to $\varphi$, and hence letting $L$ be the submodule generated by $\tilde{\varphi}$ we find that $\varphi\in [S,T]_L$ so $\varphi$ is in the image.  Finitely generated modules are isomorphic to products of terms of the form $R$ or $R/(x)$, $x\neq 0$, and so their duals will be products of terms of the form $T$ or $R/(x)$, $x\neq0$ (by Lemma \ref{biproduct} and the computations of Propositions \ref{First Full R Lemito} and \ref{Rhat isomorphism}) which both are discrete. Hence both $[S,T]$ and $\varinjlim_{L\leq M^{\wedge}} [S,T]_{L}$ are discrete so the natural bijection between them is an $\HM_R$-isomorphism.

Next, the universal property of the inverse limit tells us the natural injection $L\xrightarrow{} \varprojlim [S,T]_{L}$ is continuous. Moreover, $[S,T]_L$ is defined to make the maps $L\xrightarrow{} [S,T]_L$ surjective and hence the image of $L$ is dense. Since $L$ is fin gen it must be complete (Proposition \ref{Lemma2}), meaning that so is its image and hence the map is surjective. Hence,  $L=\varprojlim [S,T]_{L}$ as modules and so $\varprojlim [S,T]_{L}$ is also fin gen since the only topological structures on fin gen modules are canonical by Theorem \ref{fin gen uniqueness}, and so we get that the equality is both algebraic and topological.

Finally, the universal property of the projective limit implies that $M^{\wedge}\xrightarrow{}\varprojlim_{S\leq M}[S,T]$ is continuous. Moreover, this map is a bijection because as a functor to the category of sets, the representable functor $\cdot \mapsto [\cdot,T]$ turns colimits into limits (see e.g., Proposition 2.9.4 of \cite{borceux1994handbook}).

\begin{align*}
\widehat{M}&\xrightarrow{}\varprojlim_{S\leq M}[S,T]\\
&=\varprojlim_{S\leq M}\varinjlim_{L\leq M^{\wedge}} [S,T]_{L}\\
&=\varinjlim_{L\leq M^{\wedge}}\varprojlim_{S\leq M}[S,T]\\
&=\varinjlim_{L\leq M^{\wedge}}L,\\
\end{align*}

where interchange of limits is valid since limits can always be interchanged when they exist (see \cite{borceux1994handbook} Proposition 2.12.1). The module $\varinjlim_{L\leq M^{\wedge}}L$ is algebraically equal to $M^{\wedge}$ but given canonical topology, and so we get that $M^{\wedge}$ has a finer topology than the canonical topology. By Corollary \ref{Lemma3} we know that the canonical topology is the finest possible, and hence the map must be a homeomorphism. Additionally, for the compositum to be a homeomorphism we must have that $M^{\wedge}\xrightarrow{}\varprojlim_{S\leq M}[S,T]$ is a homeomorphism so we get the second claim as well.
\end{proof}

We can now prove the hard part of Theorem \ref{Topology Forcing}, as described in the introduction:

\begin{proposition}\label{Lemma6} If $M\in\HM_R$ is such that the double dual map is injective and continuous, then $M$ is canonical.
\end{proposition}
\begin{proof} Choose $M\in \LCM_R$, and let $M_0$ be the same module equipped with canonical topology. By Corollary \ref{Lemma3} we have a continuous injective identity map $M_0\xrightarrow{}M$, yielding a continuous injective precomposition map $M^{\wedge}\hookrightarrow{}M_0^{\wedge}$ by Proposition \ref{Dual continuity}. Proposition \ref{Lemma5} states that $M_0^{\wedge}$ is canonical and hence by Lemma \ref{continuous injections} so is $M^{\wedge}$, and so again by Proposition \ref{Lemma5} $M^{\wedge\wedge}$ is canonical. Thus $M$ has a continuous injection into a canonical module, and so by Lemma \ref{continuous injections} we conclude the result.
\end{proof}

To complete the proof of Theorem \ref{Topology Forcing} all we need now is injectivity and continuity, which we deduce in a final lemma:

\begin{lemma}\label{Double dual injective} The double dual map is injective and continuous for all elements of $\LCM_R$.
\end{lemma}
\begin{proof} Denote the double dual map $M\xrightarrow{}M^{\wedge\wedge}$ by $d$. The kernel of $d$ is $\bigcap_{\varphi\in M^{\wedge}}\ker\varphi$. If this intersection were nonzero then by Corollary \ref{T cogenerator} it would have a nonzero $\HM_R$-morphism into $T$. Lifting such a morphism with Proposition \ref{T almost injective} we could thus find a morphism from $M$ which does not act my zero on the intersection of kernels, which is a contradiction and hence $d$ is injective as desired.

To show that this map is continuous we chose subsets $Z\subseteq M^{\wedge}$ compact and $U\subseteq T$ open, so that $V_{M^{\wedge}}(Z,U)$ is an open set in $M^{\wedge\wedge}$, and we compute

\begin{align*}
d^{-1}\left(V_{M^{\wedge}}(Z;U))\right)&=\left\{ m\in M\st d(m)\in V_{M^{\wedge}}(Z;U)\right\}\\
&=\left\{m\in M\st d(m)(Z)\subseteq U\right\}\\
&=\left\{m\in M\st \varphi(m)\in U\,\, \forall \varphi\in Z\right\}\\
&=\bigcap_{\varphi\in Z}\varphi^{-1}(U).
\end{align*}

Choose now $m\in \bigcap_{\varphi\in Z}\varphi^{-1}(U)$, and $\mathcal{L}\ni m$ a compact neighborhood. Since $T$ is discrete by Lemma \ref{T discrete topology}, $U$ is clopen and hence so is $\varphi^{-1}(U)$. Thus, for all $\varphi\in Z$ the set $L_{m,\varphi}:=\mathcal{L}\cap \varphi^{-1}(U)$ is a compact neighborhood of $m$. The open sets $V_{M}(L_{m,\varphi},U)\ni \varphi$ cover $Z$ as $\varphi$ varies, so choosing a finite subcover we find a finite set $\varphi_1...\varphi_n$ so that

$$\bigcap_{\varphi\in Z}\varphi^{-1}(U)\supseteq \bigcap_{i=1}^{n}\bigcap_{\varphi\in V_{M}(L_{m,\varphi_i},U)}\varphi_i^{-1}(U)\supseteq \bigcap_{i=1}^{n}L_{m,\varphi_i}.$$

This last expression contains a finite intersection of open neighborhoods of $m$, and hence it is a neighborhood of $m$, so we conclude the result.
\end{proof}

\begin{proof}[Proof of Theorem \ref{Topology Forcing}] Let $M\in \LCM_R$. By Lemma \ref{Double dual injective} the double dual map is injective and continuous. By Lemma \ref{Lemma4} this means $M$ is canonical. Since the canonical topology is unique we are done.
\end{proof}

Compiling results we can now prove the main result of this text:

\begin{proof}[Proof of Theorem \ref{Equivalences theorem}] The equivalence $(1)\iff(2)$ is given in Corollary \ref{Lemma3}. Every module for which the double dual map is injective and continuous is canonical by Proposition \ref{Lemma6}. The converse of this statement is covered by Lemma \ref{Lemma4}, hence there is a unique topology for which the double dual map is injective and continuous, hence the topology in point (3) is well defined and equivalent to (1) and (2).
\end{proof}

\section{Global to Local}
\label{Global to Local}

We define $\TT_R:=\left[R,\RR/\ZZ\right]_{\ZZ}$ to be the analytic analogue of $\RR/\ZZ$ for $R$. In this section we prove Theorem \ref{Flood isomorphism theorem}, that $\TT_R\cong T$, as well as the adjointness result of \cite{flood1979pontryagin} which allows us to deduce Theorem \ref{Local Duality Theorem}. The proposition below works more generally for non-commutative rings when one keeps track of left vs right ideals, but we do not prove it here.

\begin{proposition}\label{Flood thing} Let $A$ be a locally compact topological ring. The functor $\HM_{\ZZ}\xrightarrow{} \HM_{A}$ defined by sending $M\in\HM_{\ZZ}$ to the space $[A,M]_\ZZ$ equipped with $A$-module structure $(r\cdot \varphi)(s)=\varphi(r\cdot s)$ is a right adjoint to the forgetful functor $\HM_{A}\xrightarrow{}\HM_{\ZZ}$. Even more strongly, for any $M\in \HM_{A}$ and $G\in \HM_{\ZZ}$ we have that the map

\begin{align*}
\left[M,[A,G]_{\ZZ}\right]_{A}&\xrightarrow{}[M,G]_{\ZZ}\\
\Phi &\mapsto \left(m\mapsto \Phi(m)(1)\right)
\end{align*}

is a well defined $\HM_A$-isomorphism.
\end{proposition}
\begin{proof} We denote the above defined map as $i$. Given any two $A$-modules, we use $V_{\cdot \to \cdot}(Z,U)$ to denote the subset of maps $\varphi$ between them such that $\varphi\left(Z\right)\subseteq U$. We begin by showing that $i(\Phi)$ is continuous whenever $\Phi$ is, implying that $i$ is well defined. To do this, we consider an open set $U\subseteq G$ and compute

\begin{align*}
i(\Phi)^{-1}\left(U\right)&=\left\{m\in M\st \Phi(m)(1)\in U\right\}\\
&=\left\{m\in M \st \Phi(m)\in V_{A\to G}(\{1\},U)\right\}\\
&=\Phi^{-1}\left(V_{A\to G}\left(\{1\},U\right)\right),
\end{align*}

which is open since $V_{A\to G}\left(\{1\},U\right)$ is an open set in $\left[A,G\right]_{\ZZ}$. We now show that $i$ itself is continuous. To do this, we consider additionally a compact set $Z\subseteq M$ and compute

\begin{align*}
i^{-1}\left(V_{M\to G}\left(Z,U\right)\right)&=\left\{\Phi \st \Phi(m)\in V_{M\to G}\left(Z,U\right)\right\}\\
&=\left\{\Phi \st \Phi(Z)(1)\subseteq U\right\}\\
&=\left\{\Phi \st \Phi(Z)\subseteq V_{M\to G}\left(\{1\},U\right)\right\}\\
&=V_{M\to [A,G]_{\ZZ}}\left(Z,V_{M\to G}\left(\{1\},U\right)\right)
\end{align*}

which is clearly open in the compact-open topology. To show that $i$ is an isomorphism we define the map:

\begin{align*}
[M,G]_{\ZZ}&\xrightarrow{j}\left[M,[A,G]_{\ZZ}\right]_{A}.\\
\varphi &\mapsto \left(m\mapsto \left(a\mapsto \varphi(a\cdot m)\right)\right)
\end{align*}

We now show that $j$ is a continuous inverse to $i$. Functionals in the image of $j$ are indeed continuous since they are the composition of multiplication, which is continuous by the definition of a topological module, with a continuous $\ZZ$-module homomorphism. Hence, the map is well defined. It is an $A$-module homomorphism for trivial reasons. To show that $j$ is continuous, we choose a compact set $Z\subseteq M$ and an open set $V_{A\to G}(L,U)$ in $[A,G]_{\ZZ}$ and compute

\begin{align*}
j^{-1}\left(V_{M\to [A,G]_\ZZ}(Z,V_{A\to G}(L,U))\right)&=\left\{\varphi\in [M,G]_{\ZZ}\st \varphi(L\cdot Z)\in U\right\}
\end{align*}

where $L\cdot Z$ is the set of elements of the form $a\cdot z$, $a\in A$ and $z\in Z$. For this preimage to be open it remains to show that $L\cdot Z$ is compact. Seeing as $L\cdot Z$ is the image of $L\times Z$ under the multiplication map $A\times M \xrightarrow{} M$, and that both the product and image of compact sets under continuous maps is compact, the compactness of $L\cdot Z$ is clear. Hence, $j$ is continuous.

Thus, all that remains is to show that $i$ and $j$ are inverses. This follows from the computations

\begin{align*}
(i \circ j)\left(\varphi\right)&=\left(m\mapsto \left(m\mapsto \left(a\mapsto \varphi(a\cdot m)\right)\right)(m)(1)\right)\\
&=\left(m\mapsto \left(a\mapsto \varphi(a\cdot m)\right)(1)\right)\\
&=\left(m\mapsto \varphi(m)\right)\\
&=\varphi
\end{align*}

and

\begin{align*}
(j \circ i)\left(\Phi\right)&=\left(m\mapsto \left(a\mapsto \Phi(a\cdot m)(1)\right)\right)\\
&=\left(m\mapsto \left(a\mapsto \Phi(m)(a)\right)\right)\\
&=\left(m\mapsto \Phi(m)\right)\\
&=\Phi.
\end{align*}

We have thus demonstrated the result.

\end{proof}

Clearly, we derive Theorem \ref{simplified adjointness} as a corollary.

To begin the proof of Theorem \ref{Flood isomorphism theorem} we treat the equal characteristic case, that $R\cong \FFqx$ for some finite field $\FF_q$.

\begin{lemma} For all prime powers $q=p^{e}$, we have that $[\FF_q,\RR/\ZZ]_\ZZ\cong \FF_q$ as topological $\FF_q$-modules.
\end{lemma}
\begin{proof} Seeing as $\FF_q$ is a free $\FF_p$-module, maps in $[\FF_q,\RR/\ZZ]_\ZZ$ is defined by choosing a basis and freely assigning each basis vector an element of $\frac{1}{p}\ZZ/\ZZ$. Hence $[\FF_q,\RR/\ZZ]_\ZZ$ has $q$ elements. Vector spaces ($\FF_q$ modules) are classified by their order, and hence we get that $[\FF_q,\RR/\ZZ]_\ZZ\cong \FF_q$ as an $\FF_q$ module. The topology is Hausdorff, and since the only finite Hausdorff spaces are discrete this map is a homomorphism as well so we are done.
\end{proof}

\begin{proposition}\label{equal char} Choose a prime power $q=p^{e}$. Fix an $\FF_q$-module isomorphism $[\FF_q,\RR/\ZZ]_\ZZ\xrightarrow{i} \FF_q$. We now define a map

\begin{align*}
[\FFqx,\RR/\ZZ]_\ZZ&\xrightarrow{\ell} \FF_q((x))/\FFqx.\\
\varphi&\mapsto \sum_{n=0}^{\infty}i\left(\left.\varphi\right|_{\FF_q \cdot x^n}\right) x^{-(n+1)}
\end{align*}

All but finitely many of the terms in the summation are zero. In particular, $\ell$ is well defined. Moreover, $\ell$ is an $\HM_{\FFqx}$-isomorphism.
\end{proposition}
\begin{proof} The elements in $\FF_q \cdot x^{n}$ uniformly approach zero in $\FFqx$ as $n\to\infty$. Thus, since the image under $\varphi$ of every element lies in the discrete subring $\frac{1}{p}\ZZ/\ZZ$ we must have $\varphi\left(\FF_q\cdot x^{n}\right)=0$ for all large $n$, and hence all but finitely many of the terms in the summation are zero. We now show that $\ell$ satisfies the axioms of a $\HM_{\FFqx}$-isomorphism:

\begin{itemize}
\item $\ell(\varphi_0+\varphi_1)=\ell(\varphi_0)+\ell(\varphi_1)$: This is clear, since $\varphi$ are $\ZZ$-module homomorphisms, and summation commutes with finite sums.
\item $\ell(r\cdot \varphi)=r\cdot \ell(\varphi)$: We first prove this for the case that $r=h\cdot x^{k}$, $h\in \FF_q$. Namely, we examine the following computation:

\begin{align*}
\ell\left((h\cdot x^{k})\cdot \varphi\right)&=\sum_{n=0}^{\infty}i\left(\left.((h\cdot x^{k})\cdot \varphi)\right|_{\FF_q \cdot x^n}\right) x^{-(n+1)}\\
&=\left(h\cdot x^{k}\right)\cdot \sum_{n=0}^{\infty}i\left(\left.\varphi\right|_{\FF_q \cdot x^{n+k}}\right) x^{-(n+k+1)}\\
&=\left(h\cdot x^{k}\right)\cdot \sum_{n=k}^{\infty}i\left(\left.\varphi\right|_{\FF_q \cdot x^{n}}\right) x^{-(n+1)}\\
&=\left(h\cdot x^{k}\right)\cdot \ell\left(\varphi\right).
\end{align*}

The last equality is true in $\FF_q((x))/\FFqx$ since all of the terms $n<k$ are $0$ when multiplied by $x^k$, but is certainly false in $\FF_q((x))$. Seeing as elements of the form $h\cdot x^{k}$ generate $\FFqx$ as a $\ZZ$-module we are done.

\item Injectivity: Every element of $\FF((x))/\FFqx$ can be uniquely written in the form $\sum_{n=0}^{\infty} c_n\cdot x^{-(n+1)}$. Thus, two maps sending to the same element must satisfy $\left.\varphi_0\right|_{\FF_q \cdot x^{n}}=\left.\varphi_1\right|_{\FF_q\cdot x^{n}}$ for all $n$ since $i$ is injective. Since the submodules $\FF_q \cdot x^{n}$ generate $\FFqx$ this implies $\varphi_0=\varphi_1$ and so we are done.

\item Surjectivity: Since the submodules $\FF_q \cdot x^{n}$ freely generate $\FFqx$ as a $\ZZ$-module, we can make a well defined map $\FFqx\xrightarrow{} \RR/\ZZ$ by independently choosing the behavior for each submodule $\FF_q \cdot x^{n}$. The condition that all but finitely many terms being zero is sufficient to guarantee continuity, and hence every element of $\FF_q((x))/\FFqx$ can be represented as the image of a continuous map, and hence $\ell$ is surjective.

\item Homeomorphism: By Theorem \ref{Pontryagin Duality Theorem}, the Pontryagin dual of any locally compact space is locally compact. Hence, by the above discussion, $\TT_{\FFqx}$ and $T$ are algebraiclly equal. Since $T$ is discrete by Lemma \ref{T discrete topology} and there is at most one locally compact topology on each algebraic module by Theorem \ref{Topology Forcing}, we get that $\TT_{\FFqx}$ is discrete as well so the map is a homeomorphism.
\end{itemize}

The above statements are enough to conclude that $\ell$ is an $\HM_{\FFqx}$-isomorphism, and hence we are done.
\end{proof}

Next we prove the mixed characteristic case in a non-constructive manner, remembering that $K$ is fixed as the fraction field of $R$:

\begin{lemma} \label{torsion cardinalities} If $K$ has characteristic not equal to the characteristic of the residue field of $R$, then the module $\TT_R$ is $p^{\infty}$-torsion, in the sense that each element is annihilated by $p^n$ for large enough $n$. Moreover, using $\# S$ to denote the number of elements in a set $S$ and using $M\left[p^n\right]$ to denote the $p^n$-torsion of a $\ZZ_p$-module $M$, we have that

$$\# \TT_R[p^n] = \# T[p^n]$$

for all $n\geq 0$.
\end{lemma}
\begin{proof} There is a natural $\ZZ_p$-module structure on $R$. Since $R$ is an integral domain whose fraction field has characteristic $0$ this module structure is torsion free, and since every torsion free module over a DVR is free we obtain a (non-canonical) $\ZZ_p$-module isomorphism $R\cong \ZZ_p^{e}$ where $e$ is the rank of $R$ as a $\ZZ_p$-module.  $R$ is locally compact, and hence the topology forcing theorem implies that the canonical topology on $R$ is equal to the product topology on $e$ copies of $\ZZ_p$. This computation can also be done explicitly by choosing a basis and looking at valuations, but we omit it for brevity. Finite topological products are both the product and coproduct in the category of topological modules (By Lemma \ref{biproduct}) and hence we have a $\ZZ_p$-module isomorphism

$$\TT_R\cong \left[\ZZ_p,\RR/\ZZ\right]_\ZZ^e.$$

Since $\ZZ_p$ is topologically generated by $1$, the map

\begin{align*}
\left[\ZZ_p,\RR/\ZZ\right]_\ZZ&\xrightarrow{} \RR/\ZZ\\
\varphi&\mapsto \varphi(1)
\end{align*}

is an injection. Any element $\varphi(1)$ must satisfy $\lim_{n\to\infty}p^n\varphi(1)=0$. This is true if and only if $\varphi(1)$ is a rational number with a $p$-power denominator. Rationals in such form are naturally in bijection with elements of $\QQ_p/\ZZ_p$, and the module structure from the hom agrees, and hence we get a canonical $\ZZ_p$-module isomorphism

$$\left[\ZZ_p,\RR/\ZZ\right]_\ZZ\xrightarrow{\sim}\QQ_p/\ZZ_p.$$

Combining, we see that as a $\ZZ_p$-module $\TT_R\cong \left(\QQ_p/\ZZ_p\right)^r$. It is clear from this that $\TT_R$ is $p^{\infty}$-torsion. Moreover, counting $p^n$ torsion we get that

\begin{align*}
\# \TT_R[p^n]&=\#\left(\QQ_p/\ZZ_p\right)^r[p^n]\\
&=\left(\#\left(\QQ_p/\ZZ_p\right)[p^n]\right)^r\\
&=\left(\# \ZZ/(p^n)\right)^r\\
&=p^{rn}.
\end{align*}

On the other hand, we have that

$$ \# T[p^n]=\# R/(p^n)=\left(\# R/(p)\right)^n.$$

Thus, all that is left to do for the proof is to show that $\# R/(p) = p^r$. Writing elements of $R$ uniquely in a $\ZZ_p$-basis and reducing mod $p$ this is clear, so we are done.
\end{proof}

\begin{proposition}\label{mixed char} If $K$ has characteristic not equal to the characteristic of the residue field of $R$, then there exists an $\HM_{R}$-isomorphism

$$\TT_R\xrightarrow{\sim} T$$

\end{proposition}
\begin{proof} By Proposition \ref{Flood thing} we have that

$$\left[T, \TT_R\right]_R=\left[T,\RR/\ZZ\right]_\ZZ\neq0,$$

and thus there exists a nonzero $\HM_R$-morphism $f:T\xrightarrow{} \TT_R$. Moreover, by factoring we get an injective map

$$f':T/\ker(f)\hookrightarrow{}\TT_R.$$

The submodules of $T$ are all of the form $\ker_{n}=\left\{x\,\,\mathrm{s.t}\,\, \pi^n x=0\right\}$. Multiplication by $\pi^n$ gives an isomorphism

$$T\xrightarrow{\sim} T/\ker_n,$$

and hence we can consider $f'$ as an $R$-module injection $T\hookrightarrow{} \TT_R$. To start proving surjectivity, we note that the $p^n$-torsion of $T$ must map into the $p^n$-torsion of $\TT_R$. Since both sides have the same number of elements in their $p^n$ torsions by Lemma \ref{torsion cardinalities} this means that the map is surjective onto $p^n$ torsion for each $n$. Since $\TT_R$ is a $p^\infty$-torsion module this implies surjectivity.

By Theorem \ref{Pontryagin Duality Theorem}, the Pontryagin dual of any locally compact space is locally compact. Hence, by the above discussion, $\TT_{R}$ and $T$ are algebraiclly equal. Since $T$ is discrete by Lemma \ref{T discrete topology} and there is at most one locally compact topology on each algebraic module by Theorem \ref{Topology Forcing}, we get that $\TT_R$ is discrete as well so the map is a homeomorphism. 
\end{proof}

Collecting, we can prove Theorem \ref{Flood isomorphism theorem}:

\begin{proof}[Proof of Theorem \ref{Flood isomorphism theorem}] If $R$ is a compact DVR, then either the characteristic $K$ is equal to the characteristic of the residue field or it is different. It it is equal, then by Chapter 2 §4 Theorem 2  of \cite{serre2013local} we have that $R\cong \FFqx$ for some finite field $\FF_q$, and hence this case is treated by Proposition \ref{equal char}. If the characteristic is different we can give a treatment with Proposition \ref{mixed char} so we are done.
\end{proof}

We now prove Theorem \ref{Local Duality Theorem}:

\begin{proof}[Proof of Theorem \ref{Local Duality Theorem}] As a special case of Proposition \ref{Flood thing} setting $G=\RR/\ZZ$, we get a canonical $\HM_R$-isomorphism

$$[M,\RR/\ZZ]_{\ZZ}\xrightarrow{\sim} \left[M,\TT_R\right]_{R},$$

which yields a larger canonical $\HM_R$-isomorphism

$$\left[\left[M,\RR/\ZZ\right]_{\ZZ},\RR/\ZZ\right]_{\ZZ}\xrightarrow{\sim}\left[\left[M,\TT_R\right]_{R},\TT_R\right]_{R}.$$

Next, we choose any $\HM_R$-isomorphism $i:T\xrightarrow{\sim}\TT_R$ using \ref{Flood isomorphism theorem}. This yields a map

\begin{align*}
\left[\left[M,\TT_R\right]_{R},\TT_R\right]_{R}&\xrightarrow{} \left[\left[M,T\right]_{R},T\right]_{R}.\\
\Phi &\mapsto \left(\varphi\mapsto i^{-1}\left(\Phi\left(i\circ \varphi\right)\right)\right)
\end{align*}

It is straightforward to check that this map is an $\HM_R$-isomorphism, with the inverse given by the same map in the opposite direction and $i$ replaced with $i^{-1}$. Collecting these maps and including double-dual arrows we get a diagram

\[
\begin{tikzcd}
& \left[\left[M,T\right]_R,T\right]_R\\
M \arrow{r} \arrow{ur} \arrow{dr}& \left[\left[M,\TT_R\right],\TT_R\right] \arrow{u}\\
& \left[\left[M,\RR/\ZZ\right]_\ZZ,\RR/\ZZ\right]_\ZZ .\arrow{u}
\end{tikzcd}
\]

Unpacking the definitions of the maps, it is immediate that this diagram commutes. We note that the choice of isomorphism $i$ is arbitrary, since if $\Phi(\varphi)=\varphi(m)$ then

$$ i^{-1}\left(\Phi\left(i\circ \varphi\right)\right)=i^{-1}\left(\left(i\circ \varphi\right)(m)\right)=\varphi(m)$$

as well. By Pontryagin duality (Theorem \ref{Pontryagin Duality Theorem}), the map into $\left[\left[M,\RR/\ZZ\right]_\ZZ,\RR/\ZZ\right]_\ZZ$ is a $\ZZ$-module isomorphism. It is clearly also an $R$-module morphism, and hence it is a bijective $R$-module morphism with continuous inverse and hence it is an $\HM_R$-isomorphism.

Thus, moving up the diagram and using commutativity we get that the double dual map $M\xrightarrow{} [[M,T]_R,T]_R=M^{\wedge\wedge}$ is an $\HM_R$-isomorphism so we are done.
\end{proof}

\section{Archimedean Results}\label{Completing the Proof}

In this section we prove results relating to the rings $A=\ZZ[\delta]$, where throughout $\delta$ is fixed to be the quantity $\left(b+\sqrt{b^2+4a}\right)/2$, with $a$ being a negative integer and $b$ being an integer satisfying $\left|b\right|<2\sqrt{-a}$. Firstly, the classification part of Theorem \ref{Global Flood isomorphism} is proved:

\begin{lemma}\label{Archimedean lemma 1} Let $R$ be a non-field domain, and $|\cdot|$ an archimedean absolute value on $R$. Then $R$ is complete and locally compact with respect to its topology if and only if it is isomorphic to either $\ZZ$ or $\ZZ\left[\delta\right]$ where $\delta$ is as above.
\end{lemma}
\begin{proof} The closure $\overline{K}$ of the fraction field of $R$ is a complete field with respect to an archimedean absolute value. Since every absolute value on a field of characteristic $p$ is non-archimedean, we may assume the characteristic of $\overline{K}$ is $0$. In particular, we have $\QQ \subseteq \overline{K}$. By Ostrowski's classification of absolute values on $\QQ$, we must have that $|\cdot|$ restricts to the usual absolute value and hence we have $\RR \subseteq \overline{K}$. Using this we conclude that $\overline{K}$ can be viewed as a real Banach algebra and hence by the Gelfand-Mazur theorem $\overline{K}$ is isomorphic to either $\RR$ or $\CC$.

We now have that $R$ is a subring of $\CC$. If the induced topology were not discrete, then one would have arbitrarily small elements in $R$. Using the total order of the reals this allows one to approximate every real number arbitrarily well, and hence by completeness we would have $\RR \subseteq R$, and hence $R=\RR$ or $R=\CC$ which is not possible since $R$ is assumed not a field. Hence, we have that $R$ is a discrete subring of $\CC$. Moreover, every discrete subring is complete and hence we have that $R$ satisfies the conditions of the theorem if and only if it is a discrete subring of $\CC$.

Theorem 12 of \cite{shell1996discrete} gives a classification of discrete subrings of $\CC$. Taking into account discrepancies in the definition of ring (i.e., we assume all rings are unital), this classification consists of exactly the rings listed in the statement of our lemma.
\end{proof}

Next, we prove the $\HM_A$-isomorphism part of Theorem \ref{Global Flood isomorphism}:

\begin{lemma}\label{New flood isomorphism yum} There is a canonical $\HM_A$-isomorphism

\begin{align*}
\TT_A &\xrightarrow{\sim} \CC/A\\
\varphi&\mapsto \varphi(1)+\varphi(\delta)\delta
\end{align*}
\end{lemma}
\begin{proof} The quantity $\varphi(1)+\varphi(\delta)\delta$ is well defined up to changing $\varphi(1)$ and $\varphi(\delta)$ by integer shifts, i.e., up to changing the full quantity by an element of $\ZZ[\delta]$, and hence the map's image is a well defined element of $\CC/A$. Moreover, every morphism in $\TT_A$ is determined by its action on $1$ and $\delta$ hence the above discussion implies that the map is injective.

Since $A$ is discrete, any map $A\xrightarrow{}\RR/\ZZ$ is continuous and thus the image consists of all elements in $\CC/A$. It remains to check that the map is a homeomorphism. A convergent sequence of functions in $[A,\RR/\ZZ]_{\ZZ}$ in the compact-open topology is pointwise convergent, and hence $\varphi(1)$ and $\varphi(\delta)$ converge. Moreover, $\varphi(1)$ and $\varphi(\delta)$ converging is sufficient for the a sequence of functions to be convergent since $\ZZ[\delta]$ is generated by $\varphi(1)$ and $\varphi(\delta)$ uniformly on compact sets. Hence, the topology on $[A,\RR/\ZZ]_{\ZZ}$ is the product topology on its two copies of $\RR/\ZZ$, and seeing as the topology on $\CC$ is the product topology on its two copies of $\RR$ we are done.
\end{proof}

Now, we state the final part needed for Theorem \ref{Global Flood isomorphism}:

\begin{lemma}\label{omg last lemma} The double dual map $M\xrightarrow{} \left[\left[M,\CC/A\right]_A,\CC/A\right]_A$ is an isomorphism for all locally compact $A$-modules $M$.
\end{lemma}
\begin{proof} This proof is identical to the alternate proof of Theorem \ref{Local Duality Theorem} given in section \ref{Global to Local}. The only change is that instead of using Theorem \ref{Flood isomorphism theorem} one uses Lemma \ref{New flood isomorphism yum}.
\end{proof}

We now put it all together:

\begin{proof}[Proof of Theorem \ref{Global Flood isomorphism}] This statement is simply the compositum of Lemmas \ref{Archimedean lemma 1}, \ref{New flood isomorphism yum}, and \ref{omg last lemma}. It requires no further justification.
\end{proof}

\bibliographystyle{alpha}
\bibliography{ref}

\end{document}